\documentclass{article}
\usepackage{graphicx} % Required for inserting images
\usepackage[T1]{fontenc}
\usepackage{lmodern}% Latin Modern フォントを使う

\usepackage[colorlinks=true, hidelinks]{hyperref} % ハイパーリンクを扱う拡張

\usepackage{amsmath}
\usepackage{amssymb}
\usepackage{comment}
\usepackage{cite}
\usepackage{enumerate}
\usepackage{exscale}
\usepackage{mathrsfs}
\usepackage{mathtools}
\usepackage{proof}
\usepackage{pifont}
\usepackage{qtree}
\usepackage{siunitx}
\usepackage{xcolor}
\usepackage{pifont}
\definecolor{1f1e33}{HTML}{1F1E33}
\definecolor{mediumblue}{HTML}{0000CD}

\usepackage{tikz}
\usetikzlibrary{automata, arrows.meta, positioning}

\usepackage[all]{xy}

\usepackage{amsthm}
\theoremstyle{definition}
\newtheorem{definition}{Definition}[section]
\newtheorem{example}[definition]{Example}

\newtheorem{lemma}[definition]{Lemma}
\newtheorem{theorem}[definition]{Theorem}

\newcommand{\Prop}{\mathbf{Prop}}
\newcommand{\Nom}{\mathbf{Nom}}

\newcommand{\Dia}{\Diamond}
\newcommand{\Kmodel}{\mathcal{M}}
\newcommand{\Kframe}{\mathcal{F}}
\newcommand{\tableau}{\mathbf{TAB}}

\newcommand{\hybridK}{\mathbf{K(@)}}

\newcommand{\Drest}{(\mathcal{D})}
\newcommand{\reflex}{\mathit{Ref}}
\newcommand{\irr}{\mathit{Irr}}
\newcommand{\trans}{\mathit{Trs}}
\newcommand{\ser}{\mathit{Ser}}
\newcommand{\antisym}{\mathit{A\text{-}sym}}
\newcommand{\trich}{\mathit{Tri}}
\newcommand{\total}{\mathit{Tot}}

\newcommand{\dom}{\mathrm{dom}}
\newcommand{\reflexivity}{\mathit{Eq}}
\newcommand{\Id}{\mathit{Id}}

\newcommand{\tableauIF}{\mathbf{TAB}_{\mathbf{I4}}}
\newcommand{\tableauIFD}{\mathbf{TAB}_{\mathbf{I4D}}}
\newcommand{\tableauPO}{\mathbf{TAB}_{\textbf{PO}}}
\newcommand{\tableauIFT}{\mathbf{TAB}_{\textbf{STO}}}
\newcommand{\tableauTO}{\mathbf{TAB}_{\textbf{TO}}}

\title{Terminating Hybrid Tableaus for Ordered Models}
\author{Yuki Nishimura}
\date{ }

\begin{document}

\maketitle

\begin{abstract}
    Hybrid logic extends modal logic with special propositions called nominals, each of which is true at only one state in a model. This enables us to describe some properties of binary relations, such as irreflexivity and anti-symmetry, which are essential to treat partial orders. We present terminating tableau calculi complete with respect to some classes of models whose accessibility relations are strictly partially ordered, unbounded strictly partially ordered, partially ordered, strictly totally ordered, and totally ordered.
\end{abstract}

{ \small
\begin{description}
    \item[Keywords:] Modal logic, hybrid logic, tableau calculus, decision procedure, bulldozing, partial order, total order.
\end{description}
}

\section{Introduction}

\subsection{Background and Results}

How does time flow? Prior formalized the tense in natural language using modal logic, which he named \emph{tense logic} \cite{prior1957}. Since then, logics dealing with time have been broadly studied under the name of \emph{temporal logic}. One feature of temporal logic is that transitivity is often imposed on the structure. Indeed, Goldblatt's textbook \cite{goldblatt1992} states: ``But it is natural also to require a temporal ordering to be transitive.'' In addition, temporal logics, such as LTL and CTL, handle ordered structures in reasoning about time (see also \cite{demri2016, SEPtemporal}).

Blackburn \cite{blackburn1993} extended Prior's tense logic by introducing \emph{nominal tense logic}, which incorporates \emph{nominals}---propositional variables that are true in exactly one state. This addition allows for defining properties of binary relations, such as irreflexivity, asymmetry, and anti-symmetry, which cannot be expressed in basic modal logic. The paper introduced axioms for these properties and proved the completeness and decidability of logics associated with various order structures, including partial and total orders. Notably, the logic $\mathbf{I4D}$, corresponding to models with unbounded strict partial orders, remains decidable despite lacking the finite frame property. Blackburn later named this framework \emph{hybrid logic}, which has since been extensively studied. For more details on hybrid logic, see \cite{blackburn2000, blackburn2006P, areces2007_14, indrzejczak2007, brauner2011}.

In this paper, we construct tableau calculi $\tableauIF$, $\tableauIFD$, $\tableauPO$, $\tableauIFT$, and $\tableauTO$ for models with strict partial orders, unbounded strict partial orders, partial orders, strict total orders, and total orders, respectively, and demonstrate their completeness and termination.
These five systems correspond to the logics $\mathbf{I4}$, $\mathbf{I4D}$, and $\mathbf{PO}$, $\mathbf{LIN}_s$, and $\mathbf{LIN}$ in \cite{blackburn1993}, respectively.

The key ingredient of our proof is the method called \emph{bulldozing}, which destroys \emph{undesirable states} in the model and transforms them into a desirable form. It is important to note that bulldozing transitive models can result in the construction of infinite models. However, even if the counterexample model is infinite, the existence of such a model can be demonstrated through a finite procedure. From this, we can provide a distinct proof of the decidability of the logics, as in \cite{blackburn1993}.

\subsection{Related Works}

Here, we review the previous study of tableau calculi for hybrid logic.

A tableau calculus for hybrid logic was first proposed by Tzakova \cite{tzakova1999} as a prefixed tableau calculus. The tableau calculus proposed by Bolander and Blackburn \cite{bolander2007} is a complete and terminating system for the basic hybrid logic based on $\mathbf{K}$. Their later research \cite{bolander2009} extended these results, proposing complete and terminating tableau calculi for a wider range of hybrid logics. Using this approach, we can now construct complete and terminating tableau calculi for the hybridized versions of the 15 logics belonging to the so-called modal cube \cite[\S 8]{SEPmodallogic}.

However, there are tableau calculi for other logics that have not yet been studied. In particular, the difficulty of $\mathbf{I4D}$ is believed to stem from its lack of the finite frame property. In fact, Bolander and Blackburn \cite{bolander2009} describe this challenge as follows.
\begin{quotation}
    But then we are faced with the task of combining such conditions as (irr), (sym), (asym), (antisym), (intrans), (uniq) and (tree) with (trans), and here matters are likely to be much trickier. Certainly loop-checks will usually be required, but it is unclear to us at present what kinds of general results we can hope for here, or what languages we can prove them for. To give an idea of the difficulties involved, note that even such a simple looking combination as (trans) $+$ (irr) does not have the finite frame property (consider the formula $F \top \land GF \top$, for example).
\end{quotation}

Takagi and the author \cite{nishimura2025} constructed a tableau calculus for hybrid logic corresponding to undirected graphs. To prove completeness, it was necessary to use not only traditional loop-checking but also model surgery. The methods used in that work are the basis for the proofs in this paper.

\subsection{Organization of the Paper}
The rest of this paper is organized as follows. Section \ref{secpreliminaries} reviews the semantics of hybrid logic and basic tableau calculus proposed in \cite{bolander2007}. Section \ref{secSPO} introduces a tableau calculus $\tableauIF$ for strict partially ordered models. Sections \ref{secUSPO} and \ref{secPO} give proofs of completeness and termination for tableau calculi $\tableauIFD$ and $\tableauPO$ corresponding to unbounded strict partially ordered models and partially ordered models, respectively. Section \ref{sec:TO} shows the completeness and termination for a tableau calculus $\tableauTO$ with respect to the totally ordered models and indicates the proof of that of $\tableauIFT$ corresponding to strictly totally ordered models. Section \ref{secfuture} outlines some future prospects of this research.

\section{Preliminaries}
\label{secpreliminaries}

\subsection{Kripke Semantics for Hybrid Logic}
\label{secsemantics}

Here, we review a hybrid logic with the \emph{satisfaction operators} $@_i$. See \cite{brauner2011} for more details on the semantics of hybrid logic.

\begin{definition}
  We have a countably infinite set $\Prop$ of \emph{propositional variables} and another countably infinite set $\Nom$ of \emph{nominals}, which is disjoint from $\Prop$. The \emph{formulas} $\varphi$ of hybrid logic are defined inductively as follows:
  \begin{align*}
    \varphi \Coloneqq p \mid i \mid \neg \varphi \mid \varphi \land \varphi \mid \varphi \lor \varphi \mid \Dia \varphi \mid \Box \varphi \mid @_i \varphi,
  \end{align*}
  where $p \in \Prop$ and $i \in \Nom$.

  We write $\varphi \rightarrow \psi$ to mean $\neg \varphi \lor \psi$.
\end{definition}

\begin{definition}
  A \emph{Kripke model} $\Kmodel = (W, R, V)$, or merely \emph{model}, is defined as follows:
  \begin{itemize}
    \item $W$ is a non-empty set.
    \item $R$ is a binary relation on $W$.
    \item $V$ is a function $V:\Prop \ \cup \ \Nom \to \mathcal{P}(W)$ such that $V(i) = \{ w \} \ \text{for some} \ w \in W$ for each $i \in \Nom$, where $\mathcal{P}(W)$ denotes the powerset of $W$.
  \end{itemize}

  Furthermore, we call a tuple $\Kframe = (W, R)$ that satisfies the first two conditions above a \emph{Kripke frame} (or, shortly, \emph{frame}).
\end{definition}

This definition reflects the key property of nominals: each nominal is true in only one world. This paper uses $w R v$ to mean $(w, v) \in R$, and $i^V$ to mean the world $w \in W$ such that $V(i) = \{ w \}$.

\begin{definition}
  Given a model $\Kmodel$, a possible world $w$ in $\Kmodel$, and a formula $\varphi$, the \emph{satisfaction relation} $\Kmodel, w \models \varphi$ is defined inductively as follows:
  \begin{align*}
    \Kmodel, w \models p &\iff w \in V(p), \text{where} \ p \in \Prop \\
    \Kmodel, w \models i &\iff w = i^V, \text{where} \ i \in \Nom \\
    \Kmodel, w \models \neg \varphi &\iff \text{not} \ \Kmodel, w \models \varphi \ (\Kmodel, w \not \models \varphi) \\
    \Kmodel, w \models \varphi \land \psi &\iff \Kmodel, w \models \varphi \ \text{and} \ \Kmodel, w \models \psi \\
    \Kmodel, w \models \varphi \lor \psi &\iff \Kmodel, w \models \varphi \ \text{or} \ \Kmodel, w \models \psi \\
    \Kmodel, w \models \Dia \varphi &\iff \text{there exists} \ v \ \text{such that} \ w R v \ \text{and} \ \Kmodel, v \models \varphi \\
    \Kmodel, w \models \Box \varphi &\iff \text{for all} \ v, \ \text{if} \ w R v, \ \text{then} \ \Kmodel, v \models \varphi \\
    \Kmodel, w \models @_i \varphi &\iff \Kmodel, i^V \models \varphi.
  \end{align*}
  \label{defsatis}
\end{definition}

\begin{definition}
    A formula $\varphi$ is said to be \emph{valid} (denoted by $\models \varphi$) if $\Kmodel, w \models \varphi$ holds for all models $\Kmodel$ and all of its worlds $w$.
    \label{defvalid}
\end{definition}

For simplicity, the rest of the paper deals only with the negation normal form (NNF, in short) of formulas. For the satisfaction operators, a formula $\neg @_i \varphi$ is equivalent to $@_i \neg \varphi$. That is, for any model and its possible world $w$, a formula $\varphi$, and a nominal $i \in \Nom$, we have
\[
    \Kmodel, w \models \neg @_i \varphi \iff \Kmodel, w \models @_i \neg \varphi .
\]
Transformations to the NNF involving Boolean and modal operators can be done in the usual way. In the rest of this paper, we only deal with NNF formulas. If we write $\neg \varphi$, we assume it as the abbreviation of an NNF of it.

At last, we review some properties of a binary relation.
\begin{definition}
    Given a binary relation $R$, we define the following properties:
    \begin{align*}
        R \text{ is \emph{serial} } &\iff \forall x \exists y. x R y \\
        R \text{ is \emph{reflexive} } &\iff \forall x. x R x \\
        R \text{ is \emph{irreflexive} } &\iff \forall x. \neg (x R x) \\
        R \text{ is \emph{anti-symmetric} } &\iff \forall x y. (x R y \ \& \ y R x \implies x = y) \\
        R \text{ is \emph{transitive} } &\iff \forall x y z. (x R y \ \& \ y R z \implies x R z) \\
        R \text{ is \emph{trichotomous}} &\iff \forall x y. (x R y \text{ or } y R x \text{ or } x = y) \\
        R \text{ is \emph{total}} &\iff \forall x y. (x R y \text{ or } y R x). 
    \end{align*}
    Also, we call some types of binary relations by special names:
    \begin{itemize}
        \item $R$ is a \emph{strict partial order} if it is both irreflexive and transitive.
        \item $R$ is an \emph{unbounded strict partial order} if it is serial, irreflexive, and transitive.
        \item $R$ is a \emph{partial order} if it is reflexive, anti-symmetric, and transitive.
        \item $R$ is a \emph{strict total order} if it is irreflexive, transitive, and trichotomous.
        \item $R$ is a \emph{partial order} if it is reflexive, anti-symmetric, transitive, and total.
    \end{itemize}
\end{definition}

Given a model $\Kmodel = (W, R, V)$, we say that $\Kmodel$ is irreflexive if $R$ is irreflexive. We use the same terminology for other properties of a binary relation. Moreover, we define some classes of models:
\begin{itemize}
    \item $\Kmodel$ is a \emph{strict partially ordered model}, or an \emph{SPO model}, if $R$ is a strict partial order.
    \item $\Kmodel$ is an \emph{unbounded strict partially ordered model}, or a \emph{USPO model}, if $R$ is an unbounded strict partial order.
    \item $\Kmodel$ is a \emph{partially ordered model}, or a \emph{PO model}, if $R$ is a partial order.
    \item $\Kmodel$ is a \emph{strict totally ordered model}, or an \emph{STO model}, if $R$ is a strict total order.
    \item $\Kmodel$ is a \emph{totally ordered model}, or an \emph{TO model}, if $R$ is a total order.
\end{itemize}

\subsection{Basic Tableau Calculus}
\label{secsyntaxK}

$\hybridK$ is an axiomatization of hybrid logic with an operator $@$, based on the minimal normal modal logic $\mathbf{K}$ (see \cite{blackburn2006P} for more details on the axiomatization of hybrid logic). We write $\tableau$ to indicate the tableau calculus for hybrid logic $\hybridK$. Our presentation is based on \cite{bolander2007}.

\begin{definition}
  A \emph{tableau} is a well-founded tree constructed in the following way:
  \begin{itemize}
    \item Start with a formula of the form $@_i \varphi$ (called the \emph{root formula}), where $i \in \Nom$ does not occur in the formula $\varphi$.
    \item For each branch, extend it by applying rules (see Definition \ref{defrules}). However, we can no longer add any formula in a branch if at least one of the following conditions is satisfied:
    \begin{itemize}
        \item Every new formula generated by applying any rule already exists in the branch.
        \item The branch is closed (Definition \ref{defclose} defines what a closed branch is).
    \end{itemize}
  \end{itemize}
  If a formula $\varphi$ occurs in a branch $\Theta$, we write $\varphi \in \Theta$.
  \label{deftableau}
\end{definition}

\begin{definition}
    A branch $\Theta$ of a tableau is \emph{closed} if there exists a formula $\varphi$ and a nominal $i$ such that $@_i \varphi, @_i \neg \varphi \in \Theta$. We say that $\Theta$ is \emph{open} if it is not closed. A tableau is called \emph{closed} if all branches in the tableau are closed.
    \label{defclose}
\end{definition}

\begin{definition}
  We provide the rules of $\tableau$ in Figure \ref{figrules}. In these rules, the formulas above the line show the formulas that have already occurred in the branch, and the formulas below the line show the formulas that will be added to the branch. The vertical line in the $[\lor]$ means that the branch splits to the left and right.
  \begin{figure}[tbp]
  \begin{gather*}
    \infer[{[\neg]}]
        {@_j j}
        {@_i \neg j}
    \qquad
    \infer[{[\land]}]
        {\deduce{@_i \psi} {@_i \varphi}}
        {@_i (\varphi \land \psi)}
    \qquad
    \infer[{[\lor]}]
        {@_i \varphi \mid @_i \psi}
        {@_i (\varphi \lor \psi)} \\
    \\
    \infer[{[\Dia]}^{*1, *2, *3}]
        {\deduce{@_j \varphi}{@_i \Dia j}}
        {@_i \Dia \varphi}
    \qquad
    \infer[{[\Box]}]
        {@_j \varphi}
        {\deduce{@_i \Dia j}{@_i \Box \varphi}}
    \qquad
    \infer[{[@]}]
        {@_j \varphi}
        {@_i @_j \varphi}
    \qquad
    \infer[{[\Id]}^{*3}]
        {@_j \varphi}
        {\deduce{@_i j} {@_i \varphi}}
  \end{gather*}
  
  *1: $j \in \Nom$ does not occur in the branch above.

  *2: This rule can be applied only once per formula.

  *3: The formula above the line is not an accessibility formula. Here, an \emph{accessibility formula} is the formula of the form $@_i \Dia j$ generated by [$\Dia$], where $j$ is a new nominal.
  \caption{The rules of $\tableau$}
  \label{figrules}
  \end{figure}
  \label{defrules}
\end{definition}

\begin{definition}
    Given a formula $\varphi$, we say that $\varphi$ is \emph{provable} in $\tableau$ if there is a closed tableau whose root formula is $@_i \neg \varphi$, where $i \in \Nom$ does not occur in $\varphi$.
\end{definition}

\begin{figure}[t]
  \begin{align*}
    &1. \ @_i (\Dia j \land @_j p \land \Box \neg p) & \\
    &2. \ @_i \Dia j &[\land] \\
    &3. \ @_i (@_j p \land \Box \neg p) &[\land] \\
    &4. \ @_i @_j p &[\land] \\
    &5. \ @_i \Box \neg p &[\land] \\
    &6. \ {@_i \Dia k}^* &[\Dia] \\
    &7. \ @_k j &[\Dia] \\
    &8. \ @_j p &[@] \\
    &9. \ @_k \neg p &[\Box] \\
    &10. \ @_j \neg p &[\Id] \\ 
    &\qquad \text{\ding{55}}
  \end{align*}
  \centering
  \caption{A closed tableau of $\tableau$. Formulas with ${}^*$ are accessibility formulas.}
  \label{figtab}
\end{figure}

Figure \ref{figtab} is an example of a tableau proving the hybrid formula $(\Dia j \land @_j p) \rightarrow \Dia p$. Here is an intuitive explanation of this formula: if there is a reachable world named $j$ and a formula $p$ holds in the world $j$, then there is a reachable world in which $p$ holds. It starts from $@_i (\Dia j \land @_j p \land \Box \neg p)$, whose subformula $\Dia j \land @_j p \land \Box \neg p$ is an NNF of the formula we want to prove. There are ten formulas in this branch; however, we can shorten this proof tree by applying $[\Box]$ directly to the second and fifth formulas and deriving $@_j \neg p$.

Bolander and Blackburn \cite{bolander2007} showed two significant properties of $\tableau$. One is the termination property, and the other is completeness.

\begin{theorem}
  The tableau calculus $\tableau$ has the \emph{termination property}. That is, for every formula $\varphi$, the tableau in $\tableau$ for $\varphi$ is finite.
  \label{thmterm}
\end{theorem}

\begin{theorem}
  The tableau calculus $\tableau$ is \emph{complete} for the class of all frames.
  \label{thmcomplete}
\end{theorem}

\section{Tableau Calculus for Strict Partial Order}
\label{secSPO}

In this section, we define the tableau calculus $\tableauIF$ corresponding to the SPO models.

\subsection{Definition}

First of all, we define the tableau calculus $\tableauIF$. We then give the details of the definition.

\begin{definition}
    The tableau calculus $\tableauIF$ is constructed by replacing $[\neg]$ with $[\reflexivity]$ and adding the rules $[\irr]$ and $[\trans]$, and a restriction $\Drest$ to $\tableau$.
\end{definition}

The rules $[\reflexivity]$, $[\irr]$, and $[\trans]$ are shown in Figure \ref{figrulesadd}. The restriction $\Drest$ is as follows.

\begin{description}
  \item[$\Drest$] The rule $[\Dia]$ can only be applied to a formula $@_i \Dia \varphi$ on a branch $\Theta$ if $i$ is a quasi-urfather (the \emph{quasi-urfather} is defined later in Definition \ref{defqur}).
\end{description}

We introduce the two rules $[\irr]$ and $[\trans]$ in order to express the irreflexivity and transitivity of models in the tableau calculus. By the rule $[\irr]$, we add a formula $@_i \Box \neg i$, which means ``we cannot reach the world named $i$ from the world $i$ itself.'' The rule $[\trans]$ corresponds to the fact that if a formula $\Box \varphi$ holds in a world $w$, then $\Box \varphi$ also holds in all worlds reachable from $w$.

In addition, we replace $[\neg]$ with a new rule $[\reflexivity]$. This rule allows a formula of the form $@_i i$ to be added at any time for a nominal $i$ that already appears on the branch. The rule is named after the initial letters of ``equality'' because it can be used freely, much like introducing the term $x = x$ in first-order logic. This change makes it easier to prove some lemmas.

\begin{figure}[tbp]
  \begin{gather*}
    \infer[{[\reflexivity]}^{*1}]
        {@_i i}
        { }
    \quad
    \infer[{[\ser]}^{*2, *3}]
        {@_i \Dia j}
        { }
    \quad
    \infer[{[\reflex]}^{*1, *4}]
        {@_i \Dia i}
        { }
    \infer[{[\irr]}^{*1}]
        {@_i \Box \neg i}
        { }
    \quad    
    \infer[{[\trans]}]
        {@_j \Box \varphi}
        {\deduce{@_i \Dia j}{@_i \Box \varphi}} \\
    \\
    \infer[{[\antisym]}^{*1}]
        {@_i \Box (i \lor \Box \neg i)}
        { }
    \quad
    \infer[{[\total]}^{*1, *4}]
        {@_i \Dia j \mid @_j \Dia i}
        { }
    \quad
    \infer[{[\trich]}^{*1, *4}]
        {@_i \Dia j \mid @_i j \mid @_j \Dia i}
        { }
  \end{gather*}

  *1: $i, j$ has already occurred in the branch.

  *2: $i$ is already present in the branch, while $j$ is not.

  *3: This rule can be applied only once per nominal $i$.

  *4: The formula(s) of the form $@_i \Dia j$ below the line is (are) accessibility formula(s).
  \caption{Additional rules}
  \label{figrulesadd}
\end{figure}

Then, what is the role of $\Drest$? As pointed out by \cite{bolander2009}, the addition of $[\trans]$ enables us to make an infinite branch, as shown in Figure \ref{figtabtrans}. To address the problem, we introduce a loop-checking method to add a restriction $\Drest$ to the tableau calculus. This was originally invented to deal with labeled modal tableau with inverse modalities in \cite[Section 5.2]{bolander2007}. Bolander and Blackburn also showed that this restriction works well with $\tableau + [\trans]$ in \cite[Section 7]{bolander2009}.

\begin{figure}[t]
  \begin{align*}
    &1. \ @_i (\Dia p \land \Box \Dia p) & \\
    &2. \ @_i \Dia p &[\land] \\
    &3. \ @_i \Box \Dia p &[\land] \\
    &4. \ {@_i \Dia j}^* &[\Dia] \\
    &5. \ @_j p &[\Dia] \\
    &6. \ @_j \Dia p &[\Box] \\
    &7. \ @_j \Box \Dia p &[\trans] \\
    &8. \ {@_j \Dia k}^* &[\Dia] \\
    &9. \ @_k p &[\Dia] \\ 
    &10. \ @_k \Dia p &[\Box] \\
    &11. \ @_k \Box \Dia p &[\trans] \\
    &12. \ {@_k \Dia l}^* &[\Dia] \\
    &13. \ @_l p &[\Dia] \\ 
    &14. \ @_l \Dia p &[\Box] \\
    &15. \ @_l \Box \Dia p &[\trans] \\
    &\qquad \vdots
  \end{align*}
  \caption{Non-terminating tableau of $\tableau$ with $[\trans]$. Formulas with ${}^*$ are accessibility formulas.}
  \label{figtabtrans}
\end{figure}

\vspace{\baselineskip}
To understand $\Drest$, it is necessary to prepare several concepts. Our goal is to define the concept of \emph{quasi-urfather}, introduced in \cite[\S 5.2]{bolander2007}. Intuitively, the quasi-urfather is a representative of the redundant nominals. For example, in Figure \ref{figtabtrans}, the nominals $j, k, l, \ldots$ play the same role (the worlds they point to verify the same propositions, $p$, $\Dia p$, and $\Box \Dia p$). The quasi-urfather is a representative of these redundant nominals.

\begin{definition}
  Given two formulas of the form $@_i \varphi$ and $@_j \psi$, $@_i \varphi$ is a \emph{prefixed subformula} of $@_j \psi$ if $\varphi$ is a subformula of $\psi$.
\end{definition}

\newpage

\begin{lemma}
  Let $\Theta$ be a branch of a tableau. For every formula of the form $@_i \varphi$ in $\Theta$, at least one of the following conditions is satisfied:
  \begin{itemize}
      \item It is a prefixed subformula of the root formula of $\Theta$.
      \item It is an accessibility formula.
      \item It is a prefixed subformula of $@_j \Box \neg j$ for some $j$.
  \end{itemize}
  \label{lemquasi}
\end{lemma}
\begin{proof}
    By induction on the number of applied rules. We provide the proof only for the cases $[\irr]$ and $[\trans]$. The other cases are left to the reader.
    \begin{description}
        \item[{$[\irr]$}] The formula we add by $[\irr]$ is the very formula of the form $@_j \Box \neg j$ for some $j$ that has already occurred in $\Theta$.
        \item[{$[\trans]$}] $@_j \Box \varphi$ is a prefixed subformula of $@_i \Box \varphi$. By the induction hypothesis, $@_i \Box \varphi$ is either a prefixed subformula of the root formula or the formula of the form $@_i \Box \neg k$. Therefore, $@_j \Box \varphi$ is a prefixed subformula of either the root formula or a formula $@_k \Box \neg k$. \qedhere
    \end{description}
\end{proof}

Next, we define the set $T^\Theta(i)$ for each nominal $i$ occurring in a branch $\Theta$. This set consists of the formulas in $\Theta$ that have the prefix $@_i$ and contain information related to the root formula of $\Theta$. This definition differs from the one in the proof in \cite{bolander2007}, which defines $T^\Theta(i)$ as the set $T_1^\Theta(i)$ below. This small change has a big effect on coping with both the loop-checking method and irreflexivity.

\begin{definition}
  Let $\Theta$ be a branch of a tableau. For every nominal $i$ occurring in $\Theta$, the set $T^\Theta(i)$ is a union of $T_1^\Theta(i)$ and $T_2^\Theta(i)$, where:
  \begin{itemize}
      \item $T_1^\Theta(i)$ is the set of formulas $\varphi$ such that $@_i \varphi \in \Theta$ is a prefixed subformula of the root formula of $\Theta$.
      \item $T_2^\Theta(i)$ is the set of formulas $\psi$ such that $@_i \psi \in \Theta$ is a prefixed subformula of $@_j \Box \neg j$ for some $j$ occurring in the root formula of $\Theta$.
  \end{itemize}
  \label{defT}
\end{definition}

\begin{lemma}
    For all $i$ occurring in $\Theta$, $T^\Theta(i)$ is finite.
\end{lemma}
\begin{proof}
    $T_1^\Theta(i)$ is finite because for each $i$, the number of prefixed subformulas of the root formula of $\Theta$ is finite. In turn, $T_2^\Theta(i)$ is finite because the number of nominals occurring in the root formula of $\Theta$ is finite.
\end{proof}

Using this definition, we define the concept of \emph{twins}, a tool for identifying nominals that perform the same role.

\begin{definition}
    Let $\Theta$ be a branch of a tableau. The nominals $i, j$ are \emph{twins} in $\Theta$ if $T^\Theta(i) = T^\Theta(j)$.
\end{definition}

Next, we define a generating relation between nominals. If $j$ is introduced by applying $[\Dia]$ to $@_i \Dia \varphi$, we say that $j$ is \emph{generated} by $i$. Since a new nominal is generated if and only if an accessibility formula is added, we can use the following definition.

\begin{definition}
    Let $\Theta$ be a tableau branch, and let $i$ and $j$ be nominals occurring in $\Theta$. We say that $j$ is \emph{generated} from $i$ (denoted by $i \prec_\Theta j$) if $j$ is introduced by applying $[\Dia]$ to $@_i \Dia j$.
    \label{defgenerate}
\end{definition}

In the case of $\tableau$, this definition is equivalent to the existence of an accessibility formula $@_i \Dia j$ in $\Theta$. However, later in this paper, we will deal with rules that add accessibility formulas directly without using $[\Dia]$---for example, $[\reflex]$ and $[\total]$. In such cases, this equivalence no longer holds.

Based on these preparations, we can give a precise definition of the quasi-urfather.

\begin{definition}
    We call a nominal $i$ \emph{quasi-urfather} on $\Theta$ if there are no twins $j, k$ such that $j \neq k$ and $j, k \prec_\Theta^* i$, where $\prec_\Theta^*$ denotes a reflexive and transitive closure of $\prec_\Theta$.
    \label{defqur}
\end{definition}

Now, we can understand the statement of $\Drest$. Before we show the termination property, let us show an example of a tableau whose generation terminates owing to $\Drest$.

\begin{example}
    With $\Drest$, the tableau generation in Figure \ref{figtabtrans} stops before adding the twelfth formula $@_k \Dia l$. In this branch, up to the eleventh formula $@_k \Box \Dia k$, nominals $j$ and $k$ are twins, since $T^\Theta(j) = T^\Theta(k) = \{ p, \Dia p, \Box \Dia p \}$. Then $k$ is not a quasi-urfather, so we can no longer apply $[\Dia]$ to the tenth formula $@_k \Dia p$.
\end{example}

To prevent the tableau generation from ending due to $\Drest$, it is possible to intentionally adjust the application order of the rules to avoid making nominals twins. However, such attempts will eventually fail. For example, in Figure \ref{figtabtrans2}, the restriction of $\Drest$ is circumvented by applying $[\Dia]$ to the tenth formula $@_k \Dia p$, generating a new nominal $l$ (consider that $j$ and $k$ are not twins in the branch up to the tenth formula). However, to generate further nominals, the fourteenth formula $@_l \Dia p$ is necessary, at which point $l$ ceases to be a quasi-urfather. This is because $j$ and $k$ are now twins, and $j \prec_\Theta k \prec_\Theta l$ holds.

\begin{figure}[tbp]
  \begin{align*}
    &1. \ @_i (\Dia p \land \Box \Dia p) & \\
    &2. \ @_i \Dia p &[\land] \\
    &3. \ @_i \Box \Dia p &[\land] \\
    &4. \ {@_i \Dia j}^* &[\Dia] \\
    &5. \ @_j p &[\Dia] \\
    &6. \ @_j \Dia p &[\Box] \\
    &7. \ @_j \Box \Dia p &[\trans] \\
    &8. \ {@_j \Dia k}^* &[\Dia] \\
    &9. \ @_k p &[\Dia] \\ 
    &10. \ @_k \Dia p &[\Box] \\
    &11. \ {@_k \Dia l}^* &[\Dia] \\
    &12. \ @_l p &[\Dia] \\
    &13. \ @_k \Box \Dia p &[\trans] \\ 
    &14. \ @_l \Dia p &[\Box] \\
  \end{align*}
  \caption{Another example of tableau terminated due to $\Drest$ (Formulas with ${}^*$ are accessibility formulas.)}
  \label{figtabtrans2}
\end{figure}

\subsection{Termination Property}
\label{subsecterm}

Here, we show the termination property for $\tableauIF$. The proof can be carried out based on that for the tableau calculus for $\mathbf{K4}$ in \cite{bolander2009}. One reason is that none of the added rules freely introduce formulas containing $\Dia$, and another is that the redefined $T^\Theta(i)$ remains finite.

\begin{lemma}
    Given a branch $\Theta$ of a tableau, we define a structure $G^\Theta = (N^\Theta, \prec_\Theta)$ where:
    \begin{itemize}
        \item $N^\Theta$ is the set of nominals occurring in $\Theta$.
        \item $i \prec_\Theta j$ if $j$ is generated by $i$ (see Definition \ref{defgenerate}).
    \end{itemize}
    Then, $G^\Theta$ is a finite disjoint union of well-founded and finitely branching trees.
    \label{lemG}
\end{lemma}
\begin{proof}
    See the proof of \cite[Lemma 6.4]{bolander2007}.
\end{proof}

\begin{lemma}
    Let $\Theta$ be a branch of a tableau. Then $\Theta$ is infinite if and only if we have the following infinite sequence:
    \[
        i_0 \prec_\Theta i_1 \prec_\Theta \cdots .
    \]
    \label{leminf}
\end{lemma}
\vspace{-2\baselineskip}
\begin{proof}
    The right-to-left direction is straightforward. The other direction is proved from Lemma \ref{lemG} along with K\"{o}nig's lemma. For more detail, see the proof of \cite[Lemma 6.5]{bolander2007}.
\end{proof}

\begin{theorem}
    The tableau calculus $\tableauIF$ has the termination property.
    \label{thmtermI4}
\end{theorem}
\begin{proof}
  This proof is written based on that of \cite[Theorem 5.8]{bolander2007}. We show this theorem by \textit{reductio ad absurdum}.

  Suppose there is an infinite branch $\Theta$. Then by Lemma \ref{leminf}, there is an infinite sequence of nominals as follows:
  \[
    i_0 \prec_\Theta i_1 \prec_\Theta i_2 \prec_\Theta \cdots .
  \]
  For $\Theta$ and its root formula $@_i \varphi$, we define $Q$ and $n$ as follows:
  \begin{align*}
      Q = &\{ \psi \mid \psi \text{ is a subformula of } \varphi \} \\
      \cup &\{ \psi \mid \psi \text{ is a subformula of } \Box \neg j \text{ where } j \text{ occurs in } @_i \varphi \}
  \end{align*}
  and $n$ is the number of elements in $Q$. Also, let $\Theta'$ be a fragment of $\Theta$ up to, but not including, the first occurrence of $i_{2^n + 1}$. Then $i_{2^n + 1}$ is generated by applying $[\Dia]$ to some $@_{i_{2^n}} \Dia \varphi$. Taking $\Drest$ into consideration, $i_{2^n}$ is a quasi-urfather.

  However, since all of $T^{\Theta'}(i_0), T^{\Theta'}(i_1), \ldots, T^{\Theta'}(i_{2^n})$ are subsets of $Q$ and $n$ is the cardinality of $Q$, there exists a pair $0 \leq l, m \leq 2^n$ such that $T^{\Theta'}(i_l) = T^{\Theta'}(i_m)$ by the pigeonhole principle. Therefore, we find the twins $i_l$ and $i_m$ such that $i_l, i_m \prec_\Theta^* i_{2^n}$, but that contradicts that $i_{2^n}$ is a quasi-urfather.
\end{proof}

\subsection{Completeness}

Once we construct a proof system, the next work is to show soundness and completeness. The soundness can be shown using \emph{faithful models} (see, for example, \cite{nishimura2024}). Then, we show that this tableau calculus is complete with respect to the class of SPO models.

The basic strategy is the same as in \cite[Section 5.2]{bolander2007} or \cite[Section 3.3]{nishimura2025}. That is, define \emph{identity urfathers}, representative nominals in a branch, and use them to create a model from an open saturated branch of a tableau.

\begin{definition}
    A branch $\Theta$ of a tableau is \emph{saturated} if every formula that can be generated by applying some rule already exists in $\Theta$. We call a tableau \emph{saturated} if all of its branches are saturated.
\end{definition}

If $\Theta$ is a saturated branch, the following conditions hold:
\begin{itemize}
    \item If $@_i (\varphi \land \psi) \in \Theta$, then $@_i \varphi, @_i \psi \in \Theta$.
    \item If $@_i \Dia \varphi \in \Theta$ which is not an accessibility formula and $i$ is a quasi-urfather, then there exists some $j \in \Nom$ such that $@_i \Dia j, @_j \varphi \in \Theta$.
    \item If $@_i \Box \varphi, @_i \Dia j \in \Theta$, then $@_j \varphi \in \Theta$ and $@_j \Box \varphi \in \Theta$.
    \item For all $i$ occurring in $\Theta$, we have $@_i \Box \neg i \in \Theta$.
\end{itemize}

\begin{definition}
  Let $\Theta$ be a tableau branch and $i$ a nominal occurring in $\Theta$. The \emph{identity urfather of $i$} on $\Theta$ (denoted by $v_\Theta(i)$) is the earliest introduced nominal $j$ satisfying the following conditions:
  \begin{enumerate}
    \item $j$ is a twin of $i$.
    \item $j$ is a quasi-urfather.
  \end{enumerate}
  Moreover, we say that a nominal $i$ is an \emph{identity urfather} on $\Theta$ if there is a nominal $j$ such that $v_\Theta(j) = i$.
  \label{defiur}
\end{definition}

Since not all nominals occurring in a branch have their identity urfather, we write $i \in \dom(v_\Theta)$ if $v_\Theta(i)$ exists for a nominal $i$.

We show some properties about identity urfathers. Since their proof is the same as that in \cite[Lemma 5-9]{nishimura2025}, only the statements of the lemmas are described.

\begin{lemma}
    Let $\Theta$ be a branch of a tableau. Then, the following statements hold:
    \begin{enumerate}[(a)]
        \item If $i$ occurs in the root formula $@_{i_0} \varphi_0$ of $\Theta$, then $i \in \dom(v_\Theta)$. \label{lemrightiur} 
        \item If $i$ is a quasi-urfather on $\Theta$ and $i \prec_\Theta j$, then $j \in \dom(v_\Theta)$. \label{lemiursucc}
        \item If $@_i \varphi \in \Theta$ is a prefixed subformula of the root formula of $\Theta$ and $i \in \dom(v_\Theta)$, then $@_{v_\Theta(i)} \varphi \in \Theta$. \label{lemiurcl}
        \item A nominal $i$ is the identity urfather if and only if $v_\Theta(i) = i$. \label{lemiurid}
    \end{enumerate}
    Moreover, if $\Theta$ is saturated, the following statement holds.
    \begin{enumerate}[(a)]
        \setcounter{enumi}{4}
        \item If $@_i j \in \Theta$ and $i, j \in \dom(v_\Theta)$, then $v_\Theta(i) = v_\Theta(j)$. \label{lemiureq}
    \end{enumerate}
    \label{lempropiur}
\end{lemma}

Now, we construct a model from a tableau with an open branch.   

\begin{definition}
    Given an open saturated branch $\Theta$ with a root formula $@_{i_0} \varphi_0$ of a tableau in $\tableauIF$, a model $\Kmodel^\Theta = (W^\Theta, R^\Theta, V^\Theta)$ is defined as follows:
    \begin{align*}
        W^\Theta &= \{ i \mid i \text{ is an identity urfather on } \Theta \} \\
        R_e^\Theta &= \{ (v_\Theta(i), v_\Theta(j)) \mid @_i \Dia j \in \Theta \text{ and } i, j \in \dom(v_\Theta) \} \\
        R^\Theta &= (R_e^\Theta)^+ \\
        V^\Theta(p) &= \{ v_\Theta(i) \mid @_i p \in \Theta \},\text{ where } p \in \Prop \\
        V^\Theta(i) &=
        \begin{cases}
            \{ v_\Theta(i) \} & \text{if } i \in \dom(v_\Theta) \\
            \{ i_0 \} & \text{otherwise,}
        \end{cases}
        \text{ where } i \in \Nom.
    \end{align*}
    Here, $R^+$ means the transitive closure of $R$.
    \label{defI4tableaumodel}
\end{definition}

Then we obtain the following lemma, which ensures that if we have a tableau with a root formula $@_{i_0} \varphi_0$ one of whose branches are open and saturated, there exists a model which falsifies $\varphi_0$.

\begin{lemma}
    Let $\Theta$ be an open saturated branch of a tableau in $\tableauIF$ and $@_i \varphi$ be a prefixed subformula of the root formula $@_{i_0} \varphi_0$ of $\Theta$, where $i \in \dom(v_\Theta)$. Then we have the following proposition:
    \[
        \text{if } @_i \varphi \in \Theta, \text{ then } \Kmodel^\Theta, v_\Theta(i) \models \varphi.
    \]
    In particular, $\Kmodel^\Theta, v_\Theta(i_0) \models \varphi_0$.
    \label{lemmodelexistI4}
\end{lemma}
\begin{proof}
    By induction on the complexity of $\varphi$. Most of the proof is the same as in \cite{nishimura2025}, so we only show the case of $\varphi = \Box \psi$ as it differs from the previous proof.

    Suppose that $@_i \Box \psi \in \Theta$, and take a nominal $j \in \Nom$ such that $v_\Theta(i) R^\Theta j$. (If no such nominal exists, $\Kmodel^\Theta, v_\Theta(i) \models \Box \psi$ is straightforward.) Then by the definition of $R^\Theta$, there are nominals $i_1, i_2, \ldots, i_n$ such that 
    \[
        v_\Theta(i) R_e^\Theta i_1 R_e^\Theta i_2 R_e^\Theta \cdots R_e^\Theta i_n R_e^\Theta j . 
    \]
    By the definition of $R_e^\Theta$, there are nominals $j_0, j_1, \ldots, j_n, k_1, k_2, \ldots, k_{n+1}$ such that $v_\Theta(j_0) = v_\Theta(i)$, $v_\Theta(j_m) = v_\Theta(k_m) = i_m$ for all $1 \leq m \leq n$, $v_\Theta(k_{n+1}) = j$, and 
    \[
        @_{j_0} \Dia k_1, @_{j_1} \Dia k_2, \ldots, @_{j_n} \Dia k_{n+1} \in \Theta .
    \]
    From $v_\Theta(j_0) = v_\Theta(i)$, we have $@_{j_0} \Box \psi \in \Theta$. Since $\Theta$ is saturated, we have $@_{k_1} \Box \psi \in \Theta$ by $[\trans]$. Then, $@_{j_1} \Box \psi \in \Theta$ holds since $j_1$ and $k_1$ are twins.
    Repeating this, we have $@_{j_1} \Box \psi, @_{j_2} \Box \psi, \ldots, @_{j_n} \Box \psi \in \Theta$, and finally by $[\Box]$, $@_{k_{n+1}} \psi \in \Theta$. By the induction hypothesis, we obtain that $\Kmodel^\Theta, j \models \psi$. Since we pick $j$ arbitrarily, we have $\Kmodel, v_\Theta(i) \models \Box \psi$.
\end{proof}

From the construction, the transitivity of a model $\Kmodel^\Theta$ made from an open saturated branch $\Theta$ is guaranteed. However, it is not the case that $\Kmodel^\Theta$ is always irreflexive. We can check this fact with a simple example.

\begin{example}
    Let $\Theta$ be a branch in Figure \ref{figtabsaturated}. Note that the root formula of $\Theta$ is the same as that of Figure \ref{figtabtrans}; however, it is open and saturated and terminates owing to $\Drest$ since $j$ and $k$ are twins (we can check that $T^\Theta(j) = T^\Theta(k) = \{ p, \Dia p, \Box \Dia p, \neg i, \Box \neg i \}$). From the branch, we can construct a model $\Kmodel^\Theta = (W^\Theta, R^\Theta, V^\Theta)$ as follows:
    \begin{align*}
        W^\Theta &= \{ i, j \} \\
        R^\Theta &= \{ (i, j), (j, j) \} \\
        V^\Theta(p) &= \{ j \} \\
        V^\Theta(i) &= \{ i \}.
    \end{align*}
    However, this model has the reflexive world $j$. Figure \ref{figrefpoint} shows it graphically.
    \begin{figure}[tbp]
      \begin{align*}
        &1. \ @_i (\Dia p \land \Box \Dia p) & \\
        &2. \ @_i i &[\reflexivity] \\
        &3. \ @_i \Box \neg i &[\irr] \\
        &4. \ @_i \Dia p &[\land] \\
        &5. \ @_i \Box \Dia p &[\land] \\
        &6. \ {@_i \Dia j}^* &[\Dia] \\
        &7. \ @_j p &[\Dia] \\
        &8. \ @_j j &[\reflexivity] \\
        &9. \ @_j \Box \neg j &[\irr] \\
        &10. \ @_j \neg i &[\Box] \\
        &11. \ @_j \Box \neg i &[\trans] \\
        &12. \ @_j \Dia p &[\Box] \\
        &13. \ @_j \Box \Dia p &[\trans] \\
        &14. \ {@_j \Dia k}^* &[\Dia] \\
        &15. \ @_k p &[\Dia] \\ 
        &16. \ @_k k &[\reflexivity] \\
        &17. \ @_k \Box \neg k &[\irr] \\
        &18. \ @_k \neg j &[\Box] \\
        &19. \ @_k \Box \neg j &[\trans] \\
        &20. \ @_k \neg i &[\Box] \\
        &21. \ @_k \Box \neg i &[\trans] \\
        &22. \ @_k \Dia p &[\Box] \\
        &23. \ @_k \Box \Dia p &[\trans] \\
      \end{align*}
      \caption{An open saturated branch of $\tableauIF$. Formulas with ${}^*$ are accessibility formulas.}
      \label{figtabsaturated}
    \end{figure}
    
    \begin{figure}[tbph]
        \centering
        \begin{tikzpicture}
            \node (w0) [state] {$i$};
            \node (w1) [state, right = of w0] {$j$};
            \path [->, >=stealth]
                (w0) edge (w1)
                (w1) edge [loop above] (w1);
        \end{tikzpicture}
        \caption{The model $\Kmodel^\Theta$ constructed from a branch $\Theta$ in Figure \ref{figtabsaturated}}
        \label{figrefpoint}
    \end{figure}
\end{example}

We observe that $@_k \neg j \in \Theta$ in the branch $\Theta$ in Figure \ref{figtabsaturated}. Since we create a model from an open saturated branch based on identity urfathers, we cannot distinguish twins that are not in the root formula. Then, we cannot rule out the possibility of a \emph{cluster} defined as follows being formed.

\begin{definition}
    Let $(W, R)$ be a transitive frame. A \emph{cluster} $C$ is a set such that $C \subseteq W$, $R$ over $C$ is a universal relation, and for any $D$ such that $C \subsetneq D \subseteq W$, $R$ over $D$ is not a universal relation. A cluster $C$ is \emph{simple} if $C$ has only one element, and \emph{proper} if it consists of more than one element. 
\end{definition}

To solve this problem, we bulldoze all the reflexive points and turn a frame into an irreflexive one. This method was initially used by Segerberg \cite{segerberg1971}. The procedure for bulldozing is as follows. First, we insert a strict partial order into each cluster, aligning the possible worlds in a single row. Then, we copy this row and concatenate it infinitely. As a result, all clusters are unraveled into irreflexive infinite chains.

For us to deal with hybrid logic, we have to make a slight modification to define the valuation function. Since this paper deals only with finite models, it suffices to use the set $\mathbb{N}$ of natural numbers to index clusters.

\begin{definition}
    Given a model $\Kmodel = (W, R, V)$, the bulldozed model $\Kmodel_B = (W_B, R_B, V_B)$ of $\Kmodel$ is defined as follows:
    \begin{enumerate}[1.]
        \item Index the clusters in $\Kmodel$ by $\mathbb{N}$, like $C_0, C_1, \ldots$.
        \item Insert an arbitrary strict total order $<_n$ in each cluster $C_n$.
        \item Define $C'_n = C_n \times \mathbb{N}$ (infinitely many copies of $C_n$).
        \item Define $W_B = W^- \cup \bigcup_{n \in \mathbb{N}} C'_n$, where $W^- = W \setminus \bigcup_{n \in \mathbb{N}} C_n$ is a set of irreflexive worlds.
        \item Define $\alpha: W_B \rightarrow W$ by $\alpha(w) = w$ if $w \in W^-$ and $\alpha((w, m)) = w$ otherwise.
        \item Define $R_B$. $w R_B v$ if and only if one of the following conditions holds:
        \begin{itemize}
            \item $w \in W^-$ or $v \in W^-$, and $\alpha(w) R \alpha(v)$ (A relation including irreflexive worlds is preserved.)
            \item $w \in C'_m$, $v \in C'_n$, $m \neq n$, and $\alpha(w) R \alpha(v)$ (A relation between different clusters is preserved.)
            \item $w = (w', m)$ and $v = (v', n)$ are in the same $C'_l$ and: 
            \begin{itemize}
                \item $m < n$ in the usual order $<$ of $\mathbb{N}$ or
                \item $m = n$ and $w <_l v$.
            \end{itemize}
        \end{itemize}
        \item Define $V_B: \Prop \cup \Nom \rightarrow \mathcal{P}(W_B)$ as follows:
        \begin{itemize}
            \item If $p \in \Prop$, then $w \in V_B(p)$ if and only if $\alpha(w) \in V(p)$.
            \item If $i \in \Nom$, then
            \[
                V_B(i) =
                \begin{cases}
                    \{ (i^V, 0) \} &\text{if } i^V \text{ is in some } C_n \\
                    \{ i^V \} &\text{otherwise}.
                \end{cases}
            \]
            That is, a nominal is true only at the beginning of one of the copies.
        \end{itemize}
    \end{enumerate}
    \label{defbulldoze}
\end{definition}

Observe that this procedure makes the model both irreflexive and transitive.

\begin{lemma}
    If a model $\Kmodel = (W, R, V)$ is transitive, then the bulldozed one $\Kmodel_B = (W_B, R_B, V_B)$ is irreflexive and transitive.
    \label{lembulldozing}
\end{lemma}
\begin{proof}
    First, we show that $\Kmodel_B$ is irreflexive. If $w \in W^-$, not $w R_B w$ follows from the construction. Then, it suffices to show that $w R_B w$ does not hold for any $w = (w', n) \ (w' \in C_m, n \in \mathbb{N})$ by \textit{reductio ad absurdum}. Suppose that $(w', n) R_B (w', n)$. By the definition of $R_B$ and the fact that $n = n$, we have $w' <_m w'$, which contradicts the relation $<_m$ being a strict total order. 

    Next, we show that $\Kmodel_B$ is transitive. Suppose $w R_B v$ and $v R_B u$. In the cases that not all $w, v, u$ belong to the same set $C'_n$, we can show $w R_B u$ with the transitivity of $R$. For example, let all $w, v, u$ belong to different clusters. Then we have $\alpha(w) R \alpha(v)$ and $\alpha(v) R \alpha(u)$ by assumption. Since $R$ is transitive, we have $\alpha(w) R \alpha(u)$, which leads us to the conclusion. The case to watch out for is when all $w, v, u$ belong to the same set $C'_n$. In this case, we can assume that $w = (w', a), v = (v', b), u = (u', c) \ (w', v', u' \in C_n, a, b, c \in \mathbb{N})$. Since $w R_B v$ and $v R_B u$, we have $a \leq b \leq c$. If $a \neq c$ then $a < c$, hence $w R_B u$. If $a = b = c$, we have $w' <_n v' <_n u'$ by the assumptions. Then it follows that $w R_B u$.
\end{proof}

\begin{example}
    Let $\Kmodel$ be a model illustrated in Figure \ref{figrefpoint}. This model has one (simple) cluster consisting of only one world. Then, the bulldozed model $\Kmodel_B$ in Figure \ref{figbulldozed} has an infinitely long chain of worlds, $(j, 0), (j, 1), \ldots$. Note that a nominal $j$ is true only in the world $(j, 0)$. 
\end{example}

\begin{figure}[tb]
    \centering
    \begin{tikzpicture}
        \node (w0) [state, inner sep = 8.8pt] {$i$};
        \node (w1) [state, right = of w0] {$(j, 0)$};
        \node (w2) [state, right = of w1] {$(j, 1)$};
        \node (w3) [state, right = of w2] {$(j, 2)$};
        \node (inf) [right = of w3] {$\cdots$};
        \path [->, >=stealth]
            (w0) edge (w1)
            (w1) edge (w2)
            (w2) edge (w3)
            (w3) edge (inf);
    \end{tikzpicture}
    \caption{The bulldozed model $\Kmodel_\Theta$ based on the model in Figure \ref{figrefpoint}}
    \label{figbulldozed}
\end{figure}

The bulldozing method lets us make the SPO model. However, a bulldozed model does not completely inherit the properties of the original one. For example, in the model of Figure \ref{figrefpoint}, the formula $\Dia j$ is true at the point $j$, but it cannot be true at all the points $(j, n)$ where $n \in \mathbb{N}$ in the model of Figure \ref{figbulldozed}, though every $(j, n)$ was made from the world $j$. 

However, when we construct and transform the model $\Kmodel^\Theta$, we are only interested in the truth value of the root formula of $\Theta$ and its prefixed subformulas. In such a limited situation, the bulldozing method works well.

First, we distinguish the possible worlds in a model constructed from a tableau based on whether they are related to the nominal of the root formula. As a result, it follows that all possible worlds forming a cluster are irrelevant to any nominal contained in the root formula.

\begin{lemma}
    Let $C_n$ be a cluster of $W^\Theta$. If $i \in C_n$, then every nominal $j$ such that $\Kmodel^\Theta, i \models j$ does not occur in the root formula of $\Theta$.
    \label{lemclusternotnamed}
\end{lemma}
\begin{proof}
    Since $i$ is in a cluster, we have $i R^\Theta i$.
    We assume the existence of a nominal $j$ occurring in the root formula such that $\Kmodel^\Theta, i \models j$ and derive a contradiction.

    If $C_n$ is simple, then we have $i R_e^\Theta i$. Thus, there are nominals $k, l$ such that $v_\Theta(k) = v_\Theta(l) = i$ and $@_k \Dia l \in \Theta$. Since $\Kmodel^\Theta, i \models j$ holds, we also have $v_\Theta(j) = i$, which means that all the $i, j, k, l$ are twins. Then, from $@_j j \in \Theta$ by $[\reflexivity]$, we have $@_k j, @_l j \in \Theta$. Then, we can show that $\Theta$ is closed because of it being saturated as follows: $@_k k \in \Theta$ by $[\reflexivity]$, $@_j k \in \Theta$ by $[\Id]$, $@_k \Box \neg k \in \Theta$ by $[\irr]$, $@_l \neg k \in \Theta$ by $[\Box]$, and $@_j \neg k \in \Theta$ by $[\Id]$. However, it contradicts the assumption that $\Theta$ is open. 

    Consider the other case in which $C_n$ is proper. Similarly to the former case, we have $v_\Theta(j) = i$. In this case, we have some worlds $i_1, \ldots, i_m$ such that $i R_e^\Theta i_1 R_e^\Theta \cdots R_e^\Theta i_m R_e^\Theta i$. From $i R_e^\Theta i_1$, there are $k, l_1$ such that $v_\Theta(k) = i$, $v_\Theta(l_1) = i_1$, and $@_k \Dia l_1 \in \Theta$. From these facts and the fact that $\Theta$ is saturated, we can derive $@_{i_1} \Box \neg j \in \Theta$ as follows: $@_j \Box \neg j \in \Theta$ by $[\irr]$, $@_k \Box \neg j \in \Theta$ since $j$ and $k$ are twins, $@_{l_1} \Box \neg j \in \Theta$ by $[\trans]$, and $@_{i_1} \Box \neg j \in \Theta$ since $l_1$ and $i_1$ are twins. Repeating this argument, we have $@_{i_2} \Box \neg j, \ldots, @_{i_m} \Box \neg j \in \Theta$. Finally, from $i_m R_e^\Theta i$, we obtain $@_i \neg j \in \Theta$. However, we also have $@_i j \in \Theta$ from the fact that $v_\Theta(j) = i$. Therefore, $\Theta$ is closed, which is a contradiction.
\end{proof}

From this fact, if we consider only the prefixed subformulas of the root formula, the copies of a possible world in a cluster are equivalent.

\begin{lemma}
    Let $\Theta$ be an open saturated branch. For any nominal $i \in W^\Theta$ in some cluster $C_l$, formula $\varphi$ such that $@_i \varphi$ is a prefixed subformula of the root formula, and $m, n \in \mathbb{N}$, we have
    \[
        \Kmodel^\Theta_B, (i, m) \models \varphi \iff \Kmodel^\Theta_B, (i, n) \models \varphi.
    \]
    \label{lemclusterequiv}
\end{lemma}
\vspace{-2\baselineskip}
\begin{proof}
    By induction on the complexity of $\varphi$.
    We give the proof only for the cases $\varphi = j$, $\Dia \psi$, and $\Box \psi$. The other cases are left to the reader.
    \begin{description}
        \item[{$[\varphi = j]$}] Firstly, suppose that there is a nominal $j$ such that $\Kmodel^\Theta_B, (i, 0) \models j$. Then, we have $\Kmodel^\Theta, i \models j$. Since $j$ occurs in the root formula, $i$ does not belong to any cluster $C_l$ by Lemma \ref{lemclusternotnamed}, which is a contradiction.
        Moreover, $\Kmodel^\Theta_B, (i, n) \models j$ never holds if $n \geq 1$ from the way of bulldozing. 
        Therefore, for any $n \in \mathbb{N}$, we have $\Kmodel^\Theta_B, (i, n) \not \models j$.
        \item[{$[\varphi = \Dia \psi]$}] Suppose $\Kmodel^\Theta_B, (i, m) \models \Dia \psi$. Then, there exists $w \in W^\Theta_B$ such that $(i, m) R^\Theta_B w$ and $\Kmodel^\Theta_B, w \models \psi$. If $w$ is not in $C'_l$, then $(i, n) R^\Theta_B w$ immediately follows from the construction of $R^\Theta_B$, so $\Kmodel^\Theta_B, (i, n) \models \Dia \psi$ holds. If not, that is, $w \in C'_l$, then we can assume that $w = (w', k)$. If $w$ is also reachable from $(i, n)$, there is nothing to prove. Even if not, it follows that $\Kmodel^\Theta_B, (w', n+1) \models \psi$ from the induction hypothesis. Therefore, we have $\Kmodel^\Theta_B, (i, n) \models \Dia \psi$.
        \item[{$[\varphi = \Box \psi]$}] Suppose $\Kmodel^\Theta_B, (i, m) \models \Box \psi$. Take an arbitrary $w \in W^\Theta_B$ such that $(i, n) R^\Theta_B w$, and we show that $\Kmodel^\Theta_B, w \models \psi$. If $w$ does not belong to $C'_l$, we have $(i, m) R^\Theta_B w$, and the proof ends straightforwardly. If not, assume $w = (w', k)$. Suppose $w$ is not reachable from $(i, m)$. Even in this case, we have that $(i, m) R^\Theta_B (w', m+1)$. Then $\Kmodel^\Theta_B, (w', m+1) \models \psi$ holds and $\Kmodel^\Theta_B, (w', k) \models \psi$ follows by the induction hypothesis. \qedhere
    \end{description}
\end{proof}

Now, we show the equivalence of a model from a tableau branch and the bulldozed one up to the contents of the root formula.

\begin{lemma}
    Let $\Theta$ be an open saturated branch, $i$ be an identity urfather of $\Theta$, and $\varphi$ be a formula such that $@_i \varphi$ is a prefixed subformula of the root formula. Then we have
    \[
        \Kmodel^\Theta, i \models \varphi \iff \Kmodel^\Theta_B, i_B \models \varphi ,
    \]
    where
    \[
        i_B = 
        \begin{cases}
            (i, 0) &\text{if } i \text{ is in some } C_n \\
            i &\text{otherwise}.
        \end{cases}
    \]
    \label{lembulldozingpreserve}
\end{lemma}
\begin{proof}
    By induction on the complexity of $\varphi$.
    \begin{description}
        \item[{$[\varphi = p, \neg p]$}] Straightforward.
        \item[{$[\varphi = j, \neg j]$}] In the case $\varphi = j$, we can proceed as follows:
        \begin{align*}
            \Kmodel^\Theta, i \models j &\iff V^\Theta(j) = \{ i \} \\
            &\iff V^\Theta_B(j) = \{ i_B \} \\
            &\iff \Kmodel^\Theta_B, i_B \models j .
        \end{align*}
        We can do similarly in the case $\varphi = \neg j$.
        \item[{$[\varphi = \psi \land \chi, \psi \lor \chi]$}] Straightforward.
        \item[{$[\varphi = \Dia \psi]$}] First, we show the left-to-right direction. Suppose $\Kmodel^\Theta, i \models \Dia \psi$. Then, there exists $j \in W^\Theta$ such that $i R^\Theta j$ and $\Kmodel^\Theta, j \models \psi$. If $i$ and $j$ are not in the same cluster, then by the construction of $R^\Theta_B$, we have $i_B R^\Theta_B j_B$. Also, $\Kmodel^\Theta_B, j_B \models \psi$ follows from the induction hypothesis. Hence, $\Kmodel^\Theta_B, i_B \models \Dia \psi$ holds. Suppose $i$ and $j$ are in the same cluster (it can be true that $i = j$) and $i_B R^\Theta_B j_B$ fails. In this case, though, we see that $(i, 0) R^\Theta_B (j, 1)$ holds. Also, we have $\Kmodel^\Theta_B, (j, 0) \models \psi$ by the induction hypothesis, and we obtain $\Kmodel^\Theta_B, (j, 1) \models \psi$ by Lemma \ref{lemclusterequiv}. Therefore, we have $\Kmodel, (i, 0) \models \Dia \psi$. 

        The proof of the other direction is straightforward from the fact that for any $i, j$, $i R^\Theta_B j$ implies $\alpha(i) R^\Theta \alpha(j)$. 
        \item[{$[\varphi = \Box \psi]$}] First, to show the left-to-right direction, suppose that $\Kmodel^\Theta, i \models \Box \psi$. Take one $j$ such that $i_B R^\Theta_B j$, and we show that $\Kmodel^\Theta_B, j \models \psi$.

        From $i_B R^\Theta_B j$, we have $i R^\Theta \alpha(j)$. Then, we have $\Kmodel^\Theta, \alpha(j) \models \psi$. If $\alpha(j)$ does not belong to any clusters, we have $\Kmodel^\Theta_B, j \models \psi$ by induction hypothesis. If not, $j$ must be the form $j = (j', n)$. By induction hypothesis, we have $\Kmodel^\Theta_B, (j', 0) \models \psi$. However, by Lemma \ref{lemclusterequiv}, we obtain $\Kmodel^\Theta_B, (j', n) \models \psi$.

        The other direction is like the right-to-left direction in the case $\varphi = \Dia \psi$. 
        \item[{$[\varphi = @_j \psi]$}] This proof proceeds directly. \qedhere
    \end{description}
\end{proof}

\begin{theorem}
    The tableau calculus $\tableauIF$ is complete for the class of all SPO models.
    \label{thmcompleteI4}
\end{theorem}
\begin{proof}
    We show the contraposition. Suppose $\varphi$ is not provable in $\tableauIF$. Then we can find a finite tableau in $\tableauIF$, whose root formula is $@_i \neg \varphi$, where $i$ does not occur in $\varphi$. Thus, we can make an open and saturated branch $\Theta$. Then, by Lemma \ref{lemmodelexistI4}, we have $\Kmodel^\Theta, v_\Theta(i) \models \neg \varphi$. Moreover, by Lemma \ref{lembulldozing} and Lemma \ref{lembulldozingpreserve}, we obtain an SPO model $\Kmodel^\Theta_B$ and its world $(v_\Theta(i))_B$ that falsify $\varphi$. 
\end{proof}

By utilizing the termination, soundness, and completeness of the tableau calculus, we can provide a proof for the decidability of the logic $\mathbf{I4}$ that is distinct from the one in \cite{blackburn1993}.
Given a formula $\varphi$, we can create a tableau whose root formula is $@_i \varphi$ where $i$ is not in $\varphi$.
We can make all the branches of it closed or saturated.
Owing to Theorem \ref{thmtermI4}, this procedure ends in finite time. If the tableau is closed, then $\varphi$ is provable.
If not, we can find an open saturated branch.
Then, there exists an SPO model (note that it might be infinite) that falsifies $\varphi$ by the proof of Theorem \ref{thmcompleteI4}, which means that $\varphi$ is not provable by soundness.
For the remaining systems introduced hereafter, decidability is similarly derived as a corollary.

\section{For Unbounded Strict Partial Order}
\label{secUSPO}

To deal with models with an unbounded strict partial order, we add an additional inference rule $[\ser]$. To begin with, we observe that rule $[\ser]$ reflects seriality using a simple example. Figure \ref{figtabserial} is the proof of the D axiom $\Box p \rightarrow \Dia p$ in the tableau calculus $\tableau$ with $[\ser]$.

\begin{figure}[tbp]
  \begin{align*}
    &1. \ @_i (\Box p \land \Box \neg p) & \\
    &2. \ @_i \Dia j &[\ser] \\
    &3. \ @_i \Box p &[\land] \\
    &4. \ @_i \Box \neg p &[\land] \\
    &5. \ @_j p &[\Box] \\
    &6. \ @_j \neg p &[\Box] \\
    &\quad \text{\ding{55}}
  \end{align*}
  \caption{A tableau proving the D axiom in $\tableau + [\ser]$}
  \label{figtabserial}
\end{figure}

\begin{definition}
    The tableau calculus $\tableauIFD$ is constructed by adding the rule $[\ser]$ to $\tableauIF$.
\end{definition}

The proof of the termination property and completeness can be constructed by modifying some points in that $\tableauIF$. First, for the termination proof, we change the clause of the subformula property slightly.

\begin{lemma}
  Let $\Theta$ be a branch of a tableau. For every formula of the form $@_i \varphi$ in $\Theta$, one of the following conditions holds:
  \begin{itemize}
      \item It is a prefixed subformula of the root formula of $\Theta$.
      \item It is an accessibility formula.
      \item It is a prefixed subformula of $@_j \Box \neg j$ for some $j$.
      \item It is a prefixed subformula of $@_j \Dia k$ for some $j, k$.
  \end{itemize}
  \label{lemquasiI4D}
\end{lemma}

Moreover, the addition of $[\ser]$ may affect the proof of Lemma \ref{lemG}. Here, we say that all the trees in $G^\Theta$ are finitely branching, but this may be broken by adding formulas with $\Dia$, like $@_i \Dia j$. However, $[\ser]$ only allows adding such a formula once for each nominal. Therefore, in the end, each tree in $G^\Theta$ is guaranteed to be finitely branching.

Finally, we obtain the termination property for $\tableauIFD$.

\begin{theorem}
    The tableau calculus $\tableauIFD$ has the termination property.
    \label{thmtermI4D}
\end{theorem}

The next thing to do is to show the completeness. This can also be proved in the same way as the proof for the case of $\tableauIF$, but it is necessary to show the following two additional lemmas.

\begin{lemma}
    If $\Theta$ is an open saturated branch of a tableau in $\tableauIFD$, then the model $\Kmodel^\Theta$ defined in Definition \ref{defI4tableaumodel} is serial.
\end{lemma}
\begin{proof}
    Take one $i \in W^\Theta$. Since $i$ is an identity urfather, it is also a quasi-urfather. By our assumption, $\Theta$ is saturated, so there is some $j$ such that $@_i \Dia j \in \Theta$ by $[\ser]$. Then, we have $@_i \Dia k, @_k j \in \Theta$ for some $k$ by $[\Dia]$, which implies $i \prec_\Theta k$. This fact and Lemma \ref{lempropiur} (\ref{lemiursucc}) guarantee the existence of $v_\Theta(k)$. Therefore, we have $i R^\Theta v_\Theta(k)$.
\end{proof}

\begin{lemma}
    If $\Kmodel$ is serial, then so is the bulldozed $\Kmodel_B$.
\end{lemma}
\begin{proof}
    Take one $w \in W_B$. Then, by the seriality of $R$, there is some $v' \in W$ such that $\alpha(w) R v'$. If $\alpha(w) \in W^-$, we can find some $v \in W_B$ such that $\alpha(v) = v'$ and $w R v$. If not, $\alpha(w)$ belongs to a cluster, then $w$ has the form of $(\alpha(w), n)$, where $n \in \mathbb{N}$. By the definition of $R_B$, we have $(\alpha(w), n) R_B (\alpha(w), n+1)$.
\end{proof}

Owing to these lemmas, we can obtain a USPO model which falsifies $\varphi$ for any $\varphi$ which is not provable in $\tableauIFD$, which leads us to the completeness theorem.

\begin{theorem}
    The tableau calculus $\tableauIFD$ is complete for the class of all USPO models.
    \label{thmcompleteI4D}
\end{theorem}

\section{For Partial Order}
\label{secPO}

To construct a tableau calculus corresponding to the class of PO models, we need the rules: $[\reflex]$ for reflexivity, $[\antisym]$ for anti-symmetry, and $[\trans]$ for transitivity.

Here, we briefly check the intuitive meaning of the formula $@_i \Box (i \lor \Box \neg i)$ as it may be unfamiliar to the reader. Take one world $w$ and name it $i$. Next, take another world $v$ that is reachable from but apart from $w$. Then, every world reachable from $v$ does not verify $i$; that is, it cannot be the same as $w$. This is exactly anti-symmetry; if we move to another world, we cannot go back.
\begin{definition}
    The tableau calculus $\tableauPO$ is constructed by replacing $[\neg]$ with $[\reflexivity]$, adding the rules $[\reflex]$, $[\antisym]$, and $[\trans]$, and adding the restriction $\Drest$ to $\tableau$.
\end{definition}

\subsection{Termination Property}
\label{subsectermPO}

We obtain the termination property for $\tableauPO$ by modifying the discussion in Section \ref{subsecterm}. First, observe that a subformula property in a slightly different format holds.

\begin{lemma}
  Let $\Theta$ be a branch of a tableau of $\tableauPO$. For every formula of the form $@_i \varphi$ in $\Theta$, one of the following conditions holds:
  \begin{itemize}
      \item It is a prefixed subformula of the root formula of $\Theta$.
      \item It is an accessibility formula.
      \item It is a prefixed subformula of $@_j \Box (j \lor \Box \neg j)$ for some $j$.
  \end{itemize}
  \label{lemquasiPO}
\end{lemma}

Next, we redefine the set $T^\Theta(i)$ so that it works for the tableau calculus dealing with both anti-symmetry and transitivity. Note that the sets in the definition below are also finite.
\begin{definition}
  Let $\Theta$ be a branch of a tableau. For every nominal $i$ occurring in $\Theta$, the set $T^\Theta(i)$ is a union of $T_1^\Theta(i)$ and $T_2^\Theta(i)$, where:
  \begin{itemize}
      \item $T_1^\Theta(i)$ is the set of formulas $\varphi$ such that $@_i \varphi \in \Theta$ is a prefixed subformula of the root formula of $\Theta$.
      \item $T_2^\Theta(i)$ is the set of formulas $\psi$ such that $@_i \psi \in \Theta$ is a prefixed subformula of $@_j \Box (j \lor \Box \neg j)$ for some $j$ occurring in the root formula of $\Theta$.
  \end{itemize}
  \label{defTPO}
\end{definition}

Also, we can check that the graph $G^\Theta = (N^\Theta, \prec_\Theta)$ remains a disjoint union of well-defined and finitely branching trees. The only new rule that may break this property is $[\reflex]$; however, $[\reflex]$ only allows us to add an accessibility formula, and it makes no change to the relation $\prec_\Theta$.

Then, we are ready to show the termination property.

\begin{theorem}
    The tableau calculus $\tableauPO$ has the termination property.
    \label{thmtermPO}
\end{theorem}
\begin{proof}
    For $\Theta$ and its root formula $@_i \varphi$, we define $Q$ and $n$ as follows:
    \begin{align*}
        Q = &\{ \psi \mid \psi \text{ is a subformula of } \varphi \} \\
        \cup &\{ \psi \mid \psi \text{ is a subformula of } \Box (j \lor \Box \neg j) \text{ where } j \text{ occurs in } @_i \varphi \}
    \end{align*}
    and $n$ is the number of elements in $Q$. The rest of the proof can be done like that of Theorem \ref{thmtermI4}.
\end{proof}

\subsection{Completeness}

When showing the completeness of $\tableauPO$, the process is the same; take identity urfathers, create a model, and bulldoze it. Using Definition \ref{defTPO}, we can keep adopting Definition \ref{defiur} as the definition of identity urfathers, and relative lemmas still hold. From constructing a model, we need to make some changes.

\begin{definition}
    Given an open saturated branch $\Theta$ with a root formula $@_{i_0} \varphi_0$ of a tableau in $\tableauPO$, a model $\Kmodel^\Theta = (W^\Theta, R^\Theta, V^\Theta)$ is defined as follows:
    \begin{align*}
        W^\Theta &= \{ i \mid i \text{ is an identity urfather on } \Theta \} \\
        R_e^\Theta &= \{ (v_\Theta(i), v_\Theta(j)) \mid @_i \Dia j \in \Theta \text{ and } i, j \in \dom(v_\Theta) \} \\
        R^\Theta &= (R_e^\Theta)^* \\
        V^\Theta(p) &= \{ v_\Theta(i) \mid @_i p \in \Theta \},\text{ where } p \in \Prop \\
        V^\Theta(i) &=
        \begin{cases}
            \{ v_\Theta(i) \} & \text{if } i \in \dom(v_\Theta) \\
            \{ i_0 \} & \text{otherwise,}
        \end{cases}
        \text{ where } i \in \Nom.
    \end{align*}
    Here, $R^*$ is the reflexive and transitive closure of $R$.
    \label{defPOtableaumodel}
\end{definition}
\begin{lemma}
    Let $\Theta$ be an open saturated branch of a tableau in $\tableauPO$ and $@_i \varphi$ be a prefixed subformula of the root formula $@_{i_0} \varphi_0$ of $\Theta$, where $i$ denotes an identity urfather. Then we have the following proposition:
    \[
        \text{if } @_i \varphi \in \Theta, \text{ then } \Kmodel^\Theta, v_\Theta(i) \models \varphi.
    \]
    In particular, $\Kmodel^\Theta, v_\Theta(i_0) \models \varphi_0$.
    \label{lemmodelexistS4As}
\end{lemma}
\begin{proof}
    By induction on the complexity of $\varphi$. 
    
    The only case we have to watch out for is the case $\varphi = \Box \psi$. In this case, we automatically have $v_\Theta(i) R^\Theta v_\Theta(i)$ by the reflexivity of $R^\Theta$, so we have to show that $\Kmodel^\Theta, v^\Theta(i) \models \psi$.
    By the saturation of the branch $\Theta$, $@_i \Dia i \in \Theta$ holds. Thus, by the assumption, we obtain $@_i \psi$ in $\Theta$ applying $[\Box]$. By induction hypothesis, we have $\Kmodel^\Theta, v^\Theta(i) \models \psi$.
\end{proof}

From the definition, $\Kmodel^\Theta$ is always reflexive and transitive. However, the anti-symmetry of $\Kmodel^\Theta$ is not guaranteed. Then, we bulldoze it into a PO model. 

To create not an SPO but a PO model, we redefine bulldozing by changing two points. One is that we only bulldoze proper clusters to keep a model reflexive. The other is that we insert a total order in each proper cluster instead of a strict total order.

\begin{definition}
    Given a model $\Kmodel = (W, R, V)$, the bulldozed model $\Kmodel_{B'} = (W_{B'}, R_{B'}, V_{B'})$ of $\Kmodel$ is defined as follows:
    \begin{enumerate}[1.]
        \item Index the \emph{proper} clusters in $\Kmodel$ by $\mathbb{N}$.
        \item Insert a total order $\leq_n$ in each $C_n$.
        \item Define $W_{B'}, \alpha', R_{B'}, V_{B'}$ as we do in Definition \ref{defbulldoze}.
    \end{enumerate}
    \label{defbulldozePO}
\end{definition}

\begin{lemma}
    If a model $\Kmodel = (W, R, V)$ is reflexive and transitive, then the bulldozed one $\Kmodel_{B'} = (W_{B'}, R_{B'}, V_{B'})$ is reflexive, transitive, and anti-symmetric.
    \label{lembulldozingPO}
\end{lemma}
\begin{proof}
    Reflexivity is straightforward, and transitivity can be shown as in Lemma \ref{lembulldozing}. Therefore, we show that $\Kmodel_{B'} = (W_{B'}, R_{B'}, V_{B'})$ is anti-symmetric.

    Take two worlds $w, v$ such that $w R_{B'} v$. If they do not belong to the same set $C'_n$, $v R_{B'} w$ never holds; otherwise, $\alpha(w) R \alpha(v)$ together with $\alpha(v) R \alpha(w)$ means that they are in the same cluster. Then, suppose they are in the same set $C'_n$. Then, $w$ and $v$ can be written as $w = (w', l), v = (v', m)$. From $w R_{B'} v$, we have either $l < m$ or, $l = m$ and $w' \leq_n v'$. If $v R_{B'} w$ also holds, then it follows that $l = m$ and $v' \leq_n w'$. Since $\leq_n$ is a total order, we have $w' = v'$, which implies $w = v$.
\end{proof}

The key fact in the process of proving the completeness for $\tableauIF$ is that every world in a cluster in $\Kmodel^\Theta$ is not named, and $[\irr]$ plays an important role in proving it. In the following discussion, a similar situation happens; all the worlds in proper clusters are not named owing to $[\antisym]$, and the proof proceeds based on the following lemma.

\begin{lemma}
    Let $\Theta$ be an open saturated branch of a tableau of $\tableauPO$, $\Kmodel^\Theta$ be a model constructed by $\Theta$, and $C_n$ be a proper cluster of $W^\Theta$. If $i \in C_n$, then every nominal $j$ such that $\Kmodel^\Theta, i \models j$ does not occur in the root formula of $\Theta$.
    \label{lemclusternotnamedPO}
\end{lemma}
\begin{proof}
    By \textit{reductio ad absurdum}.
    
     Suppose that we have a nominal $j$ occurring in the root formula of $\Theta$ such that $\Kmodel^\Theta, i \models j$. Since $i$ is in a proper cluster, we have some worlds $i_1, \ldots, i_m$ different from $i$ such that $i R_e^\Theta i_1 R_e^\Theta \cdots R_e^\Theta i_m R_e^\Theta i$. Moreover, we have $@_j \Box (j \lor \Box \neg j) \in \Theta$. Then, we obtain $@_{i_1} (j \lor \Box \neg j) \in \Theta$ from $i R_e^\Theta i_1$, like in the proof of Lemma \ref{lemclusternotnamed}. Now, we have two cases that either $@_{i_1} j$ or $@_{i_1} \Box \neg j$ is in $\Theta$.
     \begin{enumerate}[(a)]
         \item If $@_{i_1} j \in \Theta$, then $v_\Theta(i_1) = v_\Theta(j)$ by Lemma \ref{lempropiur} (\ref{lemiureq}). Moreover, we have $v_\Theta(j) = i$ by assumption. Therefore, $v_\Theta(i_1) = i$ holds, but it contradicts that both $i$ and $i_1$ are identity urfathers and elements in $W^\Theta$.
         \item If $@_{i_1} \Box \neg j \in \Theta$, then we obtain $@_{i_2} \Box \neg j, \ldots, @_{i_m} \Box \neg j \in \Theta$, and $@_i \neg j \in \Theta$. Therefore, $\Theta$ contains both $@_i j$ and $@_i \neg j$ and is closed, which is a contradiction. \qedhere
     \end{enumerate}
\end{proof}

Now, we can fill in the rest of the proof. 

\begin{lemma}
    Let $\Theta$ be an open saturated branch. For any nominal $i \in W^\Theta$ in some proper cluster $C_l$, formula $\varphi$ such that $@_i \varphi \in \Theta$ is a prefixed subformula of the root formula, and $m, n \in \mathbb{N}$, we have
    \[
        \Kmodel^\Theta_{B'}, (i, m) \models \varphi \iff \Kmodel^\Theta_{B'}, (i, n) \models \varphi.
    \]
    \label{lemclusterequivPO}
\end{lemma}

\begin{lemma}
    Let $\Theta$ be an open saturated branch, $i$ be an identity urfather of $\Theta$, and $\varphi$ be a formula such that $@_i \varphi$ is a prefixed subformula of the root formula. Then we have
    \[
        \Kmodel^\Theta, i \models \varphi \iff \Kmodel^\Theta_{B'}, i_{B'} \models \varphi ,
    \]
    where
    \[
        i_{B'} = 
        \begin{cases}
            (i, 0) &\text{if } i \text{ is in some } C_n \\
            i &\text{otherwise}.
        \end{cases}
    \]
    \label{lembulldozingpreservePO}
\end{lemma}

\begin{theorem}
    The tableau calculus $\tableauPO$ is complete for the class of all PO models.
    \label{thmcompleteSPAs}
\end{theorem}
\begin{proof}
    We show this theorem like Theorem \ref{thmcompleteI4} using Lemma \ref{lemmodelexistS4As}, Lemma \ref{lembulldozingPO}, and Lemma \ref{lembulldozingpreservePO}.
\end{proof}

\section{For Total Order}
\label{sec:TO}

The difference between a partial order and a total order is whether the binary relation is total. Therefore, by adding the rule $[\total]$ in Figure \ref{figrulesadd}, which corresponds to totality, to $\tableauPO$, we construct a tableau calculus corresponding to the class of TO models.

\begin{definition}
    The tableau calculus $\tableauTO$ is constructed by adding the rule $[\total]$ to $\tableauPO$.
\end{definition}

The rule $[\total]$ adds only accessibility formulas to a branch. Therefore, this rule has no effect on $T^\Theta(i)$ as defined in Definition \ref{defTPO}, nor on $G^\Theta$ as defined in Lemma \ref{lemG}. For this reason, the termination property can be proved by applying the preceding arguments directly.

\begin{theorem}
    The tableau calculus $\tableauTO$ has the termination property.
    \label{thm:termTO}
\end{theorem}

To prove completeness, we must establish two facts, as in Section \ref{secUSPO}.
First, a model generated from an open saturated branch automatically possesses totality.
Second, bulldozing preserves totality.

\begin{lemma}
    If $\Theta$ is an open saturated branch of a tableau in $\tableauTO$, then the model $\Kmodel^\Theta$ defined in \ref{defPOtableaumodel} is total.
    \label{lem:branch-model-is-total}
\end{lemma}
\begin{proof}
    Take any $i, j \in W^\Theta$.
    Since they are identity urfathers, we have $v_\Theta(i) = i$ and $v_\Theta(j) = j$ by Lemma \ref{lempropiur} (\ref{lemiurid}). 
    Since $\Theta$ is saturated, either $@_i \Dia j \in \Theta$ or $@_j \Dia i \in \Theta$ holds for $[\total]$.
    Therefore, we have either $i R^\Theta j$ or $j R^\Theta i$.
\end{proof}

\begin{lemma}
    If $\Kmodel$ is total, then so is the bulldozed $\Kmodel_{B'}$.
    \label{lem:bulldoze-preserves-totality}
\end{lemma}
\begin{proof}
    Take any $w, v \in \Kmodel_{B'}$.
    If they do not belong to the same cluster $C'_n$, either $w R_{B'} v$ or $v R_{B'} w$ follows from the totality of $\Kmodel$; we have either $\alpha'(w) R \alpha'(v)$ or $\alpha'(v) R \alpha'(w)$.
    If there is a cluster $C'_n$ such that $w, v \in C'_n$, either $w R_{B'} v$ or $v R_{B'} w$ follows from the definition of $R_{B'}$.
\end{proof}

\begin{theorem}
    The tableau calculus $\tableauTO$ is complete for the class of all TO models.
    \label{thm:completeTO}
\end{theorem}

Although we omit the detailed proof, we are now able to handle the class of STO models using the tableau calculus. Let us define the tableau calculus $\tableauIFT$ by adding the rule $[\trich]$ in Figure \ref{figrulesadd} to $\tableauIF$. Then, following the same line of argument as in this section, we can show that this tableau calculus possesses the termination property and is both sound and complete with respect to the class of STO models.

\section{Conclusion and Future Work}
\label{secfuture}

In this paper, we constructed tableau calculi $\tableauIF$, $\tableauIFD$, $ \tableauPO$, $\tableauIFT$, and $\tableauTO$ for hybrid logic and demonstrated their completeness and termination properties. Now, we have tableau calculi for models where possible worlds form some types of order: strict partial order, unbounded strict partial order, partial order, strict total order, and total order, respectively.

One approachable work is to generalize the methods introduced here to the broader classes of models. In the five tableau calculi discussed in this paper, we extended the definition of $T^\Theta(i)$ to reconcile transitivity with other axioms in the completeness proof. This approach is likely to extend to formulas that do not contain $\Dia$ (referred to as \emph{$F$-free formulas} in \cite{bolander2009}). Specifically, the following axioms are of interest:
\begin{itemize}
    \item asymmetry: $@_i \Box \Box \neg i$
    \item intransitivity: $@_i (\Box \Box \neg j \lor \Box \neg j)$
    \item tree-likeness: $@_i (@_j \Box \neg i \lor @_k \Box \neg i \lor @_j k)$
    \item uniqueness: $@_i (\Box j \lor \Box \neg j)$.
\end{itemize}
Combining these with transitive or symmetric models would allow us to construct tableau calculi corresponding to various classes of models\footnote{It goes without saying that $[\trans]$ and intransitivity cannot coexist.}.

Another prospective project is the application to the proof of the finite model property.
By utilizing the completeness proof of the tableau presented here, we may be able to demonstrate the finite model property for hybrid logics corresponding to classes of models with ordered structures.

In the case of $\mathbf{I4}$, we demonstrate the completeness with respect to the class of SPO models using bulldozing.
This technique potentially allows a single cluster to generate an infinitely long path, so the finite model property is not guaranteed.
However, if we work on the hybrid logic with nominals and a satisfaction operator, the modal depth of any given formula $\varphi$ is finite.
Then, there seems to be no problem in cutting off an infinitely extending path at an appropriate point (though the exact timing needs careful investigation).

The finite model properties for these logics for ordered models have already been proven by \cite{blackburn1993} using filtration. 
However, if we can show the finite model property using the strategy indicated above, it would constitute a constructive alternative proof.

\section*{Acknowledgement}
I would like to thank Prof. Ryo Kashima and Leonardo Pacheco for their invaluable advice in writing this paper. This work was supported by JST SPRING, Japan Grant Number JPMJSP2106 and JPMJSP2180.
\newpage

\bibliographystyle{plain}
\bibliography{logic}

\end{document}